\documentclass[a4paper,11pt]{article}

\usepackage{lineno}

\usepackage{graphicx}
\usepackage{tabu} 
\usepackage{geometry}
\usepackage{hyperref}
\usepackage{xcolor}
\usepackage{enumerate}
\usepackage{complexity}
\usepackage{blkarray}
\usepackage{wrapfig}
\usepackage{todonotes}
\usepackage{paralist}
\usepackage{listings}
\usepackage{varwidth}
\usepackage{comment}

\usepackage[font=small,labelfont=bf]{caption}
\usepackage[font=small,labelfont=normalfont,labelformat=simple]{subcaption}

\usepackage{listings}
\usepackage{xspace}

\graphicspath{{figs/}}

\usepackage{amssymb}
\usepackage{amsmath}
\usepackage{mathtools}

\DeclareMathOperator{\conv}{conv}
\DeclareMathOperator{\sgn}{sgn}

\def\typeI{type~I\xspace}
\def\TypeI{Type~I\xspace}
\def\typeIa{type~I$_\mathrm{a}$\xspace}
\def\TypeIa{Type~I$_\mathrm{a}$\xspace}
\def\typeIb{type~I$_\mathrm{b}$\xspace}
\def\typeII{type~II\xspace}

\def\XX{\mathcal{X}}
\def\RR{\mathbb{R}}
\def\AA{\mathcal{A}}
\def\FF{\mathcal{F}}

\usepackage{amsthm}
\newtheorem{theorem}{Theorem}

\newtheorem{lemma}[theorem]{Lemma}
\newtheorem{corollary}[theorem]{Corollary}
\newtheorem{question}{Question}

\newtheorem{proposition}{Proposition}

\date{}

\title{Topological Drawings meet Classical Theorems from Convex~Geometry\footnote{%
We thank Alan Arroyo, Emo Welzl, Heiko Harborth, and Geza T\'oth for inspiring discussions and the reviewers for helpful comments.
A special thanks goes to Patrick Schnider
for his simplification of the construction 
in the proof of Proposition~\ref{theorem:helly_fconvex}.
R.\ Steiner and H.\ Bergold were funded by DFG-GRK~2434. 
S.\ Felsner and M.\ Scheucher were supported by the DFG Grant FE~340/12-1.
M.\ Scheucher was supported by the DFG Grant SCHE~2214/1-1 and
by the internal research funding ``Post-Doc-Funding'' from Technische Universit\"at Berlin.
}
\footnote{The reverse direction of Theorem~\ref{theorem:kirchberger} does not hold. 
The comment directly after the theorem is now corrected and a counterexample is presented. }
}

\def\inst#1{$^{#1}$}

\begin{document}

\author{
Helena Bergold\inst{1}
\and
Stefan Felsner\inst{2}
\and
Manfred Scheucher\inst{2}
\and
Felix Schr\"oder\inst{2}
\and
Raphael Steiner\inst{2}
}

\maketitle

\begin{center}
{\footnotesize
\inst{1} 
Fachbereich Mathematik und Informatik, \\
Freie Universit\"at Berlin, Germany,\\
\texttt{helena.bergold@fu-berlin.de}
\\\ \\
\inst{2} 
Institut f\"ur Mathematik, \\
Technische Universit\"at Berlin, Germany,\\
\texttt{\{felsner,scheucher,fschroed,steiner\}@math.tu-berlin.de}
\\\ \\
}
\end{center}


\begin{abstract}
In this article we discuss classic theorems 
from Convex Geometry 
in the context of topological drawings and beyond.
In a simple topological drawing of the complete graph~$K_n$, 
any two edges share at most one point:
either a common vertex or a point where they cross.
Triangles of simple topological drawings can be
viewed as convex sets. This gives a link to convex geometry.

As our main result, we present a generalization of Kirchberger's theorem that is of purely combinatorial nature. It turned out that this classic theorem also applies to ``generalized signotopes'' -- a combinatorial generalization
of simple topological drawings, which we introduce and investigate in the course of this article. 
As indicated by their name they are
a generalization of signotopes, a structure studied in the context of encodings for arrangements of pseudolines.

We also present a family of simple topological drawings
with arbitrarily large Helly number, and a new proof of a topological
generalization of Carath\'{e}odory's theorem in the plane and discuss further
classic theorems from Convex Geometry in the context of simple
topological drawings.

\end{abstract}




\section{Introduction}
\label{sec:Introduction}

A set of $n$ points in the plane (in general position)
induces a straight-line drawing of the complete graph~$K_n$.
In this article we investigate simple topological drawings of $K_n$
and use triangles spanned by 3~points of such drawings to generalize and study classic
problems from the convex geometry of point sets.
Since we only deal with simple topological drawings we omit the
attribute \emph{simple} and define
a \emph{topological drawing} $D$ of $K_n$ in the plane (or on the sphere) as follows:

\goodbreak{}
\begin{itemize}
\item 
vertices are mapped to distinct points in the plane (on the sphere), 
\item
edges are mapped to simple
curves connecting the two corresponding vertices
and containing no other vertices,
and 
\item
every pair of edges has at most one
common point, which is either a common vertex or a crossing
(but not a touching).
\end{itemize}
Figure~\ref{fig:topological_drawing} shows 
the forbidden patterns for topological drawings. 
We also assume throughout the article 
that no three edges cross in a single point.
Topological drawings 
are also  known as ``good drawings'' or ``simple drawings''.

\begin{figure}[tb]
  \centering
   \includegraphics{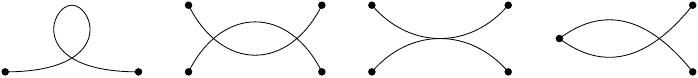}
  \caption{Forbidden patterns in topological drawings:
  self-crossings, double-crossings, touchings, and crossings of adjacent edges.}
  \label{fig:topological_drawing}  
\end{figure}

We discuss classic theorems such as Kirchberger's, Helly's,
and Carath\'eodory's theorem in terms of the \emph{convexity} hierarchy of topological drawings
developed by Arroyo, McQuillan, Richter, and
Salazar~\cite{ArroyoMRS2017_convex}, which we introduce in
Section~\ref{sec:Preliminaries}.  In that section, we
also define \emph{generalized signotopes}, a combinatorial generalization of
topological drawings.  
The connection between generalized signotopes and
topological drawings is deferred to Section~\ref{sec:GS_counting}.
Our proof of a generalization of Kirchberger's theorem
in Section~\ref{sec:Kirchberger} makes use of this combinatorial
structure. 
Section~\ref{sec:Caratheodory} deals with a generalization of
Carath\'eodory's theorem. In Section~\ref{sec:Helly}, we present a family of
topological drawings with arbitrarily large Helly number. 
We conclude this article with Section~\ref{sec:Discussion},
where we discuss some open problems.

\section{Preliminaries}
\label{sec:Preliminaries}

Let $D$ be a topological drawing and $v$ a vertex of $D$. The cyclic order $\pi_v$ of incident edges
around $v$ is called the \emph{rotation} of $v$
in $D$. 
The collection of rotations of all vertices is called the \emph{rotation system} of~$D$. 
Two topological drawings are \emph{weakly isomorphic} 
if there is an isomorphism of the underlying abstract graphs 
which preserves the rotation system or reverses all rotations.
Note that all projections of a drawing on the sphere to a plane are weakly isomorphic.

A triangular cell, which has no vertex on its boundary, is bounded by three
edges.  By moving one of these edges across the intersection of the two other
edges, one obtains a weakly isomorphic drawing; see Figure
\ref{fig:weak_strong_isomorphism}.  This operation is called
\emph{triangle-flip}. Gioan~\cite{Gioan05_full}, see also Arroyo et
al.~\cite{ArroyoMRS2017_reidemeister}, showed that any two weakly isomorphic
drawings of the complete graph can be transformed into each other with a
sequence of triangle-flips and at most one reflection of the drawing.

\begin{figure}[tb]
	\centering
	\includegraphics{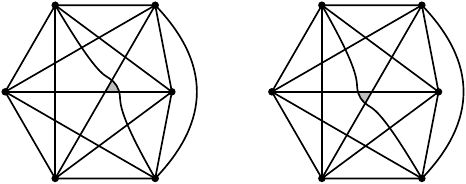}
	\caption{Two weakly isomorphic drawings of $K_6$,
	which can be transformed into each other by a triangle-flip.}
	\label{fig:weak_strong_isomorphism}
\end{figure}

Besides weak isomorphism, there is also the notion of
strong isomorphism: two topological drawings are called \emph{strongly isomorphic}
if they induce homeomorphic cell decompositions of the sphere.

\subsection{Convexity Hierarchy} 
\label{ssec:convhier}

Given a topological drawing $D$, 
we call the induced subdrawing of three vertices a \emph{triangle}.
Note that the edges of a triangle in a topological drawing do not cross.
The removal of a triangle separates the plane into two connected components
-- a bounded component and an unbounded component. 
We call the closure of these connected components \emph{sides}.
A side of a triangle is \emph{convex}
if every edge that has its two end-vertices in the side
is completely drawn in the side.
We are now ready to introduce the 
``convexity hierarchy'' of Arroyo et al.~\cite{ArroyoMRS2017_convex}).
For $1 \leq i < j \leq 6$, drawings with property ($j$)
also have property ($i$).

\begin{enumerate}[\quad (1)]
	\item topological drawings;
	
	\item \emph{convex} drawings:
	each triangle has a convex side;
	
	\item \emph{hereditary-convex} drawings:
	we can choose a convex side for each triangle such that,
	if a triangle~$\triangle_1$ is fully contained 
	in the chosen convex side of another triangle~$\triangle_2$,
	then also its chosen convex side is;
	
	\item \emph{cell-convex}\footnote{
		The authors of \cite{ArroyoMRS2017_convex} use the term \emph{face-convex} instead of cell-convex.
	} drawings: 
	there is a special cell $c_\infty$  
	such that, for every triangle, the side not containing $c_\infty$ is convex; 
	
	\item \emph{pseudolinear} drawings:
	there is an arrangement $\AA$ of pseudolines such that each
	edge of the drawing is supported by (contained in) one of the pseudolines of~$\AA$\footnote{
		Arrangements supporting a drawing of $K_n$ are also known as
		\emph{pseudoconfigurations of points}
		and can be considered as oriented matroids of rank~3 
		(cf.\  Chapter~5.3 of \cite{BjoenerLVWSZ1993}).
		For a formal definition of arrangements of pseudolines, 
		we refer the interested reader to \cite{FelsnerGoodman2016} or Chapter~6 of \cite{BjoenerLVWSZ1993}.
	};
	
	\item straight-line drawings:
	all edges are drawn as straight-line segments connecting their endpoints.
\end{enumerate}

Arroyo et al.~\cite{ArroyoMRS2017_pseudolines} showed that
the cell-convex drawings where the special cell~$c_\infty$ is drawn as the unbounded outer cell
are precisely the pseudolinear drawings
(see also~\cite{ArroyoBR2018} and \cite{AichholzerHPSV2015}).

\medskip

Pseudolinear drawings are generalized by {pseudocircular} drawings.
A drawing is called \emph{pseudocircular} if the edges can be extended to pseudocircles
(simple closed curves) such that any pair of pseudocircles
either has two crossings or is disjoint. Since stereographic projections preserve (pseudo)circles,
pseudocircularity is a property of drawings on the sphere.
Pseudocircular drawings were studied in a recent article by Arroyo, Richter,
and Sunohara \cite{ArroyoRS2021}. They provided an example of a 
topological drawing which is not pseudocircular.
Moreover, they proved that hereditary-convex drawings are precisely the
\emph{pseudospherical} drawings, i.e., pseudocircular drawings with the
additional two properties that
\begin{itemize}
 \item every pair of pseudocircles intersects, and
 \item for any two edges $e\neq f$ the pseudocircle $\gamma_e$
   has at most one crossing with~$f$.
\end{itemize}
The relation between convex drawings and 
pseudocircular drawings remains open.

\medskip

Convexity, hereditary-convexity, and cell-convexity are properties of the weak
isomorphism classes.  To see this, note that the existence of a convex side is
not affected by changing the outer cell or by transferring the drawing to the
sphere, moreover, convex sides are not affected by triangle-flips.
Hence, these properties only depend on the rotation system of the drawing.
For pseudolinear and straight-line drawings, however, the choice of the outer
cell plays an essential role.

\subsection{Generalized Signotopes}
\label{ssec:prelim_GS}

Let $D$ be a topological drawing of a complete graph in the plane. 
Assign an \emph{orientation} $\chi(abc) \in \{+,-\}$ to each ordered triple $(a,b,c)$ of vertices. 
The sign $\chi(abc)$ indicates whether 
we go counterclockwise or clockwise around the triangle
if we traverse the edges $(a,b),(b,c),(c,a)$ in this order. 

\smallskip

If $D$ is a straight-line drawing of $K_n$, then
the underlying point set  
has to be in general position (no three points are on a line).
Assuming that the points are labeled $1,\ldots,n$
and sorted from left to right,
then $\chi$ is \emph{monotone on 4-tuples}, that is,
\begin{itemize}
\item
for all $i<j<k<l$
the sequence 
$
\chi(ijk), \chi(ijl), \chi(ikl), \chi(jkl)
$
(index-triples in lexicographic order) has at most one sign-change. 
\end{itemize}
A \emph{signotope} is a mapping $\chi:\binom{[n]}{3} \to \{+,-\}$ with
the above monotonicity property, where $[n]=\{1,2,\ldots,n\}$.
Signotopes are in bijection with Euclidean pseudoline arrangements \cite{FelsnerWeil2001} 
and can be used to characterize pseudolinear
drawings \cite[Theorem 3.2]{BalkoFulekKyncl2015}.

When considering topological drawings of the complete graph
we have no meaningful ordering of the vertices.
Exchanging the labels of two vertices reverts the orientation of all
triangles containing both vertices. This suggests to look at the 
\emph{alternating} extension of $\chi$. Formally
$\chi(i_{\sigma(1)},i_{\sigma(2)},i_{\sigma(3)}) = \sgn(\sigma) \cdot \chi(i_1,i_2,i_3)$ 
for any distinct labels $i_1,i_2,i_3$ and any permutation $\sigma \in S_3$.
This yields a mapping $\chi: [n]_3 \to \{+,-\}$, where $[n]_3$ 
denotes the set of all triples $(a,b,c)$  with pairwise distinct $a,b,c \in [n]$.
To see whether the alternating extension of $\chi$ still has
a property comparable to the monotonicity of signotopes, we
have to look at 4-tuples of vertices, i.e., at drawings of $K_4$.
On the sphere there are two types of drawings of $K_4$: 
\typeI has one crossing and \typeII has no crossing. 
\TypeI can be drawn in two different ways in the plane:
in \typeIa the crossing is only incident to bounded cells 
and in \typeIb the crossing lies on the outer cell;
see Figure~\ref{fig:k4_three_types}.

\begin{figure}[tb]
	\centering
	
	\includegraphics[page=1,width=\textwidth]{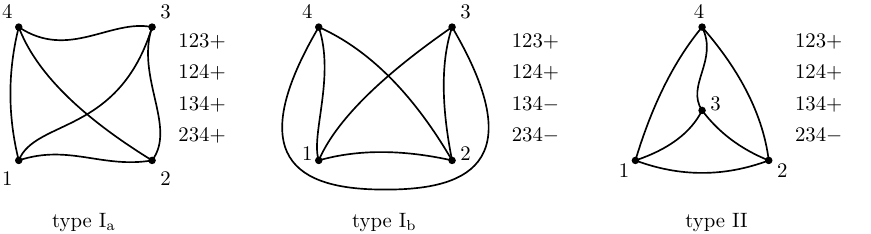}
	
	\caption{The three types of topological drawings of $K_4$ in the plane.}
	\label{fig:k4_three_types}
\end{figure}

A drawing of $K_4$ with vertices $a,b,c,d$ can be characterized 
in terms of the sequence of orientations $\chi(abc),\chi(abd),\chi(acd),\chi(bcd)$.
The drawing is 
\begin{itemize}
	\item 
	of \typeIa or \typeIb iff the sequence 
	is $++++$, $++--$, $+--+$, $-++-$, $--++$, or $----$;
	and
	
	\item 
	of \typeII 
	iff the number of $+$'s (and $-$'s respectively) in the sequence is odd.
\end{itemize}

\goodbreak 
Therefore there are at most two sign-changes
in the sequence $\chi(abc),\chi(abd),\allowbreak{}\chi(acd),\allowbreak{}\chi(bcd)$ and, moreover, 
any such sequence is in fact induced by a topological drawing of~$K_4$.
Allowing up to two sign-changes is equivalent to forbidding the two
patterns $+-+-$ and $-+-+$. 
Note that this classification is independent of the labeling and order of the vertices $a,b,c,d$. 

If  $\chi$ is alternating and
avoids the two patterns $+-+-$ and $-+-+$ on sorted indices, i.e.,
$\chi(ijk),\chi(ijl),\chi(ikl),\chi(jkl)$ has at most two
sign-changes for all $i<j<k<l$, 
then it avoids the two patterns in
$\chi(abc),\chi(abd),\chi(acd),\chi(bcd)$ for any pairwise
distinct $a,b,c,d \in [n]$. 
We refer to this as the
\emph{symmetry property} of the forbidden patterns. 

The symmetry property allows us to define \emph{generalized signotopes}
as alternating mappings \mbox{$\chi\colon[n]_3 \to \{+,-\}$} with at
most two sign-changes on $\chi(abc),\chi(abd),\allowbreak{}\chi(acd),\chi(bcd)$
for any pairwise different $a,b,c,d \in [n]$.
We conclude:

\begin{proposition}
\label{proposition:topdraw_induces_gensig}
Every topological drawing of $K_n$ induces a generalized signotope on $n$ elements.
\end{proposition}

We defer the structural investigation of generalized signotopes to Section~\ref{sec:GS_counting}, where we show that
there are more generalized signotopes
than topological drawings. Hence generalized signotopes extend the convexity hierarchy introduced above.

\section{Kirchberger's Theorem}
\label{sec:Kirchberger}

Two closed sets $A, B \subseteq \mathbb{R}^d$ are called
\emph{separable} if there exists a hyperplane $H$ separating them, 
i.e., $A\subset H_1$ and
$B\subset H_2$ with $H_1$, $H_2$ being the two closed half-spaces defined by~$H$. 
It is well-known that, if two non-empty compact sets $A, B$ are separable, then they can also be
separated by a hyperplane $H$ containing points of $A$ and $B$. 
\emph{Kirchberger's theorem} (see~\cite{Kirchberger1903} or~\cite{Barvinok2002})
asserts that two finite point sets $A, B \subseteq \mathbb{R}^d$ are
separable if and only if for every $C \subseteq A \cup B$ with
$|C|=d+2$, $C \cap A$ and $C \cap B$ are separable.

Goodman and Pollack \cite{GoodmanPollack1982} proved duals of Kirchberger's
theorem and further theorems like Radon's, Helly's, and Carath\'eodory's theorem
for arrangements of pseudolines.  Their results also transfer to
pseudoconfigurations of points and thus to pseudolinear drawings.
To be more precise, they proved a natural generalization of Kirchberger's theorem to pseudoline-arrangements in the plane which, by duality, is equivalent to a separating statement on pseudo\-configurations of points in the plane (cf.\ Theorem~4.8 and Remark~5.2 in \cite{GoodmanPollack1982}).

The 2-dimensional version of Kirchberger's theorem can be formulated in terms of triple orientations which indicate whether a point lies on the right or left side of a chosen line.  We show a generalization for topological drawings using generalized signotopes.
Two sets $A, B \subseteq [n]$ are \emph{separable} if there
exist $i,j \in A \cup B$ such that $\chi(i,j,x) = +$ for all
$x \in A \setminus\{i,j\}$ and $\chi(i,j,x)=-$ for all
$x \in B\setminus\{i,j\}$.  In this case we say that $ij$
\emph{separates} $A$ from~$B$ and write $\chi(i,j,A)=+$ and
$\chi(i,j,B)=-$.  Moreover, if we can find $i\in A $ and $j \in B$, 
we say that $A$ and $B$ are \emph{strongly separable}. 
As an example, consider the \mbox{4-element} generalized signotope of the \typeIb drawing of $K_4$ in Figure~\ref{fig:k4_three_types}. 
The sets $A=\{1,2\}$ and $B=\{3,4\}$
are strongly separable with $i=2$ and $j=3$ because $\chi(2,3,1)=+$ and $\chi(2,3,4)=-$.

\begin{theorem}[Kirchberger for Generalized Signotopes]
	\label{theorem:kirchberger}
	Let $\chi:[n]_3 \to \{+,-\}$ be a generalized signotope, and
	let $A, B \subseteq [n]$ be two non-empty sets.  
	If for every $C \subseteq A \cup B$ with
	$|C| = 4$, the sets $A \cap C$ and $B \cap C$ are separable, then
	$A$ and $B$ are strongly separable.
\end{theorem}

Note that, since every topological drawing yields a generalized signotope,
Theorem~\ref{theorem:kirchberger} generalizes Kirchberger's theorem
to topological drawings of complete graphs. In terms of topological drawings separability means that there exists an edge $ij$ such that all triangles $ija$ for  $a \in A$ are oriented counterclockwise and all triangles $ijb$ for $b \in B$ are oriented clockwise. 
In the classic version of Kirchberger's theorem, the converse statement of 
Theorem~\ref{theorem:kirchberger} is trivially true. 
A separating hyperplane for the point set separates all subsets.
However, in the setting of generalized signotopes the reverse direction is no longer true.
A separating pair of points for $A,B$ is not necessarily contained in a 4-element subset~$C \subset A \cup B$.
In Figure~\ref{fig:kirchberger_reverse_orientation}, we provide a generalized signotope on 6 elements with a separator for a fixed partition into $A = \{1,2\}$ and $B = \{3,4,5,6\}$. However for the subset $C = \{2,4,5,6\}$ the two sets $A \cap C = \{2\}$ and $B \cap C = \{4,5,6\}$ are not separable. 
Moreover, this generalized signotopes comes from a simple drawing, which is drawn in Figure~\ref{fig:kirchberger_reverse_fig}.
The edge marked bold separates the blue from the red vertices. However, the subdrawing of the $K_4$ marked with dashed edges has no separator.

\begin{figure}[htb]
	\begin{subfigure}[b]{.5\textwidth}
		\centering
		\includegraphics{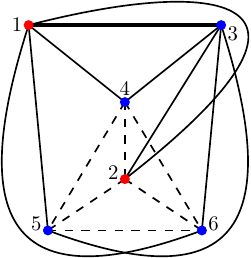}
		\vspace{3ex}
		\caption{ }
		\label{fig:kirchberger_reverse_fig}
	\end{subfigure}
	\begin{subfigure}[b]{.4\textwidth}
		\centering
		\begin{align*}
			&\gamma(4,5,6) &&= &&+  &&& \gamma(1,5,6) &&= &&- \\
			&\gamma(3,5,6) &&= &&-  &&& \gamma(1,4,6) &&= &&- \\
			&\gamma(3,4,6) &&= &&-  &&& \gamma(1,4,5) &&= &&- \\
			&\gamma(3,4,5) &&= &&-  &&& \gamma(1,3,6) &&= &&- \\
			&\gamma(2,5,6) &&= &&-  &&& \gamma(1,3,5) &&= &&- \\
			&\gamma(2,4,6) &&= &&-  &&& \gamma(1,3,4) &&= &&- \\
			&\gamma(2,4,5) &&= &&+  &&& \gamma(1,2,6) &&= &&- \\
			&\gamma(2,3,6) &&= &&-  &&& \gamma(1,2,5) &&= &&- \\
			&\gamma(2,3,5) &&= &&-  &&& \gamma(1,2,4) &&= &&- \\
			&\gamma(2,3,4) &&= &&+  &&& \gamma(1,2,3) &&= &&- 
		\end{align*}
		\caption{ }
		\label{fig:kirchberger_reverse_orientation}
	\end{subfigure}
	\caption{\subref{fig:kirchberger_reverse_fig} Simple drawing showing that the reverse direction of Kirchberger is not true. The bold edge is a separator for the drawing on all 6 vertices. However, the subdrawing of the $K_4$ marked with the dashed edges has no separator. The vertices of $A$ are marked red and the vertices of $B$ blue.
		\subref{fig:kirchberger_reverse_orientation} Orientations of the drawing yielding the generalized signotope $\gamma$.}
	\label{fig:kirchberger_reverse}
\end{figure}

\begin{proof}[Proof of Theorem~\ref{theorem:kirchberger}]
	First we prove that all 4-tuples $C \subseteq A \cup B$ with
	$C\cap A$ and \mbox{$C \cap B$} non-empty 
	which are separable are also strongly separable. 
    This can be verified looking at Tables~\ref{table:separators_1_3} and~\ref{table:separators_2_2}, which
    show that, in all separable generalized signotopes 
    on $\{a,b_1,b_2,b_3\}$ and $\{a_1,a_2,b_1,b_2\}$, respectively,
    there is a strong separator of the sets $\{a\}$ and $\{b_1,b_2,b_3\}$ or $\{a_1,a_2\}$ and $\{b_1,b_2\}$, respectively.
    Hence in the following we
	assume that all such 4-tuples from $A \cup B$ are strongly separable.

    \begin{table}[htb]
    \[
    \begin{tabu}{llll|l}
    \chi(a,b_1,b_2)&
    \chi(a,b_1,b_3)&
    \chi(a,b_2,b_3)&
    \chi(b_1,b_2,b_3)&
    \text{list of separators}\\
    \hline
    +&    +&    +&    +& \underline{ab_3},b_1a,b_1b_3\\
    +&    +&    +&    -& \underline{ab_3},b_1a,b_1b_2,b_2b_3\\
    +&    +&    -&    +& \underline{ab_2},b_1a,b_1b_3,b_3b_2\\
    +&    +&    -&    -& \underline{ab_2},b_1a,b_1b_2\\
    +&    -&    +&    +& \text{(no separator)}\\
    +&    -&    -&    +& \underline{ab_2},b_3a,b_3b_2\\
    +&    -&    -&    -& \underline{ab_2},b_1b_2,b_3a,b_3b_1\\
    -&    +&    +&    +& \underline{ab_3},b_1b_3,b_2a,b_2b_1\\
    -&    +&    +&    -& \underline{ab_3},b_2a,b_2b_3\\
    -&    +&    -&    -& \text{(no separator)}\\
    -&    -&    +&    +& \underline{ab_1},b_2a,b_2b_1\\
    -&    -&    +&    -& \underline{ab_1},b_2a,b_2b_3,b_3b_1\\
    -&    -&    -&    +& \underline{ab_1},b_2b_1,b_3a,b_3b_2\\
    -&    -&    -&    -& \underline{ab_1},b_3a,b_3b_1\\
    \end{tabu}
    \]
    \caption{Separators for generalized signotopes on $\{a,b_1,b_2,b_3\}$.
    Strong separators are underlined.} 
    \label{table:separators_1_3}
    \end{table}

    \begin{table}[htb]
    \[
    \begin{tabu}{llll|l}
    \chi(a_1a_2b_1)&
    \chi(a_1a_2b_2)&
    \chi(a_1b_1b_2)&
    \chi(a_2b_1b_2)&
    \text{list of separators}\\
    \hline
    +&   +&   +&   +& a_2a_1,\underline{a_2b_2},b_1a_1,b_1b_2\\
    +&   +&   +&   -& a_2a_1,\underline{a_2b_1},b_1a_1\\
    +&   +&   -&   +& a_2a_1,\underline{a_2b_2},b_2a_1\\
    +&   +&   -&   -& a_2a_1,\underline{a_2b_1},b_2a_1,b_2b_1\\
    +&   -&   +&   +& \underline{a_1b_2},b_1a_1,b_1b_2\\
    +&   -&   -&   +& \text{(no separator)}\\
    +&   -&   -&   -& \underline{a_2b_1},b_2a_2,b_2b_1\\
    -&   +&   +&   +& \underline{a_2b_2},b_1a_2,b_1b_2\\
    -&   +&   +&   -& \text{(no separator)}\\
    -&   +&   -&   -& \underline{a_1b_1},b_2a_1,b_2b_1\\
    -&   -&   +&   +& a_1a_2,\underline{a_1b_2},b_1a_2,b_1b_2\\
    -&   -&   +&   -& a_1a_2,\underline{a_1b_2},b_2a_2\\
    -&   -&   -&   +& a_1a_2,\underline{a_1b_1},b_1a_2\\
    -&   -&   -&   -& a_1a_2,\underline{a_1b_1},b_2a_2,b_2b_1\\
    \end{tabu}
    \]
    \caption{Separators for generalized signotopes on $\{a_1,a_2,b_1,b_2\}$.
    Strong separators are underlined.} 
    \label{table:separators_2_2}
    \end{table}

	\medskip
        
	By symmetry we may assume $|A| \leq |B|$.
	First we consider the cases $|A|=1,2,3$ individually and
	then the case $|A| \ge 4$.
	
	\medskip
	
	Let $A = \{a\}$, let
	$B'$ be a maximal subset of $B$ such that $B'$ is strongly separated from
	$\{a\}$, and let $b \in B'$ be such that $\chi(a,b,B') = -$.
	Suppose that $B' \neq B$, then there is a $b^\ast \in B \backslash B'$ with
	\begin{align}
	\label{eq:A=1_abb*=+}
	\chi(a,b,b^\ast) = +.
	\end{align}
	By maximality of $B'$ we cannot use the pair $a,b^\ast$ for a strong separation of
        $\{a\}$ and $B' \cup \{b^\ast\}$. Hence, for some $b' \in B'$:
	\begin{align}
	\label{eq:A=1_ab*b'=+}
	\chi(a,b^\ast,b') = +.
	\end{align}  
	Since $\chi$ is alternating (\ref{eq:A=1_abb*=+}) and
	(\ref{eq:A=1_ab*b'=+}) together imply $b' \neq b$.
	Since $b'\in B'$ we have $\chi(a,b,b') =-$.
	From this together with (\ref{eq:A=1_abb*=+}) and
	(\ref{eq:A=1_ab*b'=+})
	it follows that the four-element set 
	$\{a,b,b',b^\ast\}$ has no separator.
	This is a contradiction, whence $B' = B$.
	
	\medskip
	
	\goodbreak
	\noindent
	As a consequence we obtain:
	\begin{itemize}
		\item Every one-element set $\{a\}$ with $a\in A$ can be strongly separated
		from~$B$. Since $\chi$ is alternating there is a unique
		$b(a)\in B$ such that $\chi(a,b(a),B) = -$.
	\end{itemize}
	
	\medskip
	
        \noindent
	Now consider the case that $A=\{a_1,a_2\}$.  
	Let $b_i = b(a_i)$, i.e.,
	$\chi(a_i,b_i,B)  =-$ for~$i=1,2$.
	If $\chi(a_1,b_1,a_2) = +$ or if $\chi(a_2,b_2,a_1)=+$,
	    then $a_1b_1$ or $a_2b_2$, respectively,
        is a strong separator for $A$ and~$B$. 
    Therefore, we may assume that
        $\chi(a_1,b_1,a_2) = -$,
        $\chi(a_2,b_2,a_1 ) = -$ and therefore $b_1 \neq b_2$.  
        We get the sequence $+--+$ for the four-element set $\{a_1,a_2, b_1, b_2\}$
        which has no strong separator (cf. Table \ref{table:separators_2_2}), a contradiction.
	
	\medskip
	
	Let $A = \{a_1,a_2,a_3\}$. Suppose that $A$
	is not separable from $B$. Let $b_i = b(a_i)$, i.e.,
	$\chi(a_i,b_i,B)  = -$ for $i=1,2,3$.
	For $i,j\in\{1,2,3\}$, $i\neq j$ we define
	$s_{ij} = \chi(a_i,b_i, a_j)$.
	
	If $s_{ij} = +$ for some $i$ and all $j \neq i$,
        then $a_ib_i$ separates $A$ from $B$.
        Hence, for each $i$ there exists $j \neq i$ with
        $s_{ij} = -$.
	
	If $s_{ij} = s_{ji} = -$ for some $i,j$, 
	then since $\chi$ is alternating $b_i\neq b_j$ and
	$\{a_i,a_j,b_i,b_j\}$ corresponds to the row $+--+$ in
        Table~\ref{table:separators_2_2}, i.e., there is no strong separator.
	Hence, at least one of $s_{ij}$ and $s_{ji}$ is $+$.
	
	These two conditions imply that we can relabel the elements of $A$ such
	that $s_{12} = s_{23} = s_{31} = +$ and
	$s_{13} = s_{21} = s_{32} = -$. Suppose that $b_i = b_j = b$ for some $i \neq j \in \{1,2,3\}$, then the four elements 
	$\{b,a_1,a_2,a_3\}$ have the pattern $-+-*$. Avoiding the forbidden pattern, 
	we get $-+--$ 
	in Table~\ref{table:separators_1_3}, i.e., there is no strong separator. This contradiction shows that $b_1,b_2,b_3$ must be pairwise distinct. 
	
	From $s_{32} = -$ and $s_{31}=+$ we find that
        $\{b_3,a_1,a_2,a_3\}$ corresponds to a row of type $*+-*$ in
        Table~\ref{table:separators_1_3}. We conclude that the
        strong separator of $\{b_3,a_1,a_2,a_3\}$ is $a_2b_3$. In particular,
	\begin{align}
	\chi (b_3,a_1,a_2) = +. \label{eq:A=3_a2b3a1=+}
	\end{align}
	
        Now consider $\{a_1,a_2,b_1,b_3\}$. 
	From
        $s_{12} = +$, equation~(\ref{eq:A=3_a2b3a1=+}), and $\chi(a_1,b_1,b_3) = -$
        we obtain the pattern $-+-*$. Since $-+-+$ is forbidden 
	we obtain
	\begin{align}
	\chi(a_2,b_1,b_3) =-. \label{eq:A=3_b1a2b3=+}
	\end{align}
	The set $\{a_2,a_3,b_1,b_3\}$ needs a strong separator. 
	The candidate pair
	$a_3b_1$ is made impossible by $\chi (a_3,b_1,b_3) = +$,
	$a_3b_3$ is made impossible by $s_{32}=-$, and
	$a_2b_3$ is made impossible by (\ref{eq:A=3_b1a2b3=+}).
	Hence $a_2b_1$ is the strong separator and, in particular, it holds
	\begin{align}
	\chi (a_2,b_1,a_3) = + \label{eq:A=3_a2b1a3=+}.
	\end{align}
	But now the set $\{a_1,a_2,a_3,b_1\}$ has no strong separator. 
	The candidate pair $a_1b_1$ is impossible because of $s_{13} =-$,
	$a_2b_1$ does not separate because $s_{12} = +$,
	and (\ref{eq:A=3_a2b1a3=+}) shows that $a_3b_1$ cannot separate the set. 
	This contradiction proves the case $|A| =3$. 
	
	\medskip
	
	For the remaining case $|A|\geq 4$ consider a counterexample
	$(\chi,A,B)$ minimizing the size of the smaller of the two sets. 
	We have $4 \leq |A| \leq |B|$.
	
	Let $a^{\ast} \in A$. By minimality $A' = A\backslash \{a^{\ast}\}$
    is separable from $B$.  Let $a\in A'$ and $b\in B$ such that
        $\chi(a,b,A') = +$ and $\chi(a,b,B) = -$.  Hence
	\begin{align}
	\label{eq:aba*=-}
	\chi(a,b,a^{\ast}) = -.
	\end{align}
	Let $b^{\ast} = b(a^{\ast})$, i.e., $\chi(a^{\ast},b^{\ast}, B )= -$.
	There is some 
	$a' \in A'$ such that
	\begin{align}
	\label{eq:a*b*a'=-}
	\chi(a^{\ast},b^{\ast},a') = -.
	\end{align}
	
	If $a' = a$, then $b \neq b^\ast$ because of (\ref{eq:aba*=-})
        and (\ref{eq:a*b*a'=-}). 
        From (\ref{eq:aba*=-}), (\ref{eq:a*b*a'=-}), $\chi(a,b,B) = -$, and 
        $\chi(a^{\ast},b^{\ast}, B )= -$ it follows that the  four-element set
        $\{ a , a^\ast, b, b^\ast \}$ has the sign pattern $+--+$, hence there is no
        separator, see Table \ref{table:separators_2_2}. This shows that $a' \neq a$.

	Let $b' = b(a')$.  	
	If $b \neq b'$ we
	look at the four elements $\{a,b,a',b'\}$. It corresponds to
        $+-*-$ so that we can conclude $\chi(a,a',b') = -$.
        If $b = b'$, then $a' \in A'$ implies
	$\chi(a,b,a') = +$ which yields $\chi(a',b',a) = -$. 
        
	Hence, regardless whether $b=b'$ or $b\neq b'$ we have
	\begin{align}
	\label{eq:a'b'a=-}
	\chi(a',b',a) = - \ .
	\end{align}
	
	Since $|A| \geq 4$, we know by the minimality of the instance $(\chi,A,B)$ that the set
	$\{a,b, a',\allowbreak{}b',a^{\ast},b^{\ast}\}$, which has 3 elements of~$A$ and at
	least 1 element of~$B$, is separable.  It follows from
	$\chi(a,b,B) = \chi(a',b',B) = \chi(a^{\ast},b^{\ast},B) = -$ that
	the only possible strong separators are $ab$, $a'b'$, and
	$a^{\ast}b^{\ast}$. They, however, do not separate because of
	(\ref{eq:aba*=-}), (\ref{eq:a*b*a'=-}) and (\ref{eq:a'b'a=-})
	respectively.  This contradiction shows that there is no counterexample. 
\end{proof}

\section{Carath\'eodory's Theorem}
\label{sec:Caratheodory}

\emph{Carath\'eodory's theorem} asserts that,
if a point $x$ lies in the convex hull of a point set~$P$ in~$\RR^d$,
then $x$ lies in the convex hull of at most $d+1$ points of~$P$.

As already mentioned in Section~\ref{sec:Kirchberger}, Goodman and Pollack \cite{GoodmanPollack1982} proved a dual of Carath\'eodory's
theorem, which transfers to pseudolinear drawings.

A more general version of Carath\'eodory's theorem in the plane is due to Balko,
Fulek, and Kyn\v{c}l, who provided a generalization to topological drawings.
In this section, we present a shorter proof for their theorem.

\begin{theorem}[{Carath\'eodory for Topological Drawings \cite[Lemma~4.7]{BalkoFulekKyncl2015}}]
\label{theorem:Caratheodory_generalized}
Let $D$ be a topological drawing of $K_n$ and 
let $x \in \RR^2$ be a point contained in a bounded connected component of $\RR^2-D$. 
Then there is a triangle in $D$ that contains $x$ in the bounded side.
\end{theorem}

Note that, in the classic version for points in $\RR^2$, the case $|A| = 1$ of Kirchberger's theorem implies Carath\'eodory's theorem, and vice versa. This is not true for the generalized versions. The vertex $1$ of \typeIb in Figure~\ref{fig:k4_three_types} is in the triangle spanned by $2,3,4$. However, partitioning the four vertices into $A = \{1\}$ and $B = \{2,3,4\}$ gives a separating pair $(1,2)$ because the triples $(1,2,3)$ and $(1,2,4)$ are oriented the same way.

\begin{proof}[Proof of Theorem~\ref{theorem:Caratheodory_generalized}]
Suppose towards a contradiction 
that there is a pair $(D,x)$ violating the claim.
We choose~$D$ minimal with respect to the number of vertices~$n$.

Let $a$ be a vertex of the drawing.  If we remove all incident edges
of $a$ from~$D$, then, by minimality of the example, $x$ becomes a point
of the outer cell.  Therefore, if we remove the incident
edges of~$a$ one by one, we find a last subdrawing~$D'$ such that $x$
is still in a bounded cell.  Let $ab$ be the edge such that in the
drawing $D'- ab$ the point $x$ is in the outer cell.

There is a simple curve $P$ 
connecting $x$ to infinity, which does not cross any of the edges in $D'-ab$.
By the choice of~$D'$, the curve $P$~has at least one crossing with $ab$. 
We choose $P$ minimal with respect to the number of crossings with $ab$. 

We claim that $P$ intersects $ab$ exactly once. 
Suppose that $P$ crosses $ab$ more than once. Then there 
is a \emph{lense}~$C$ formed by $P$ and $ab$, 
that is, two crossings of $P$ and $ab$ 
such that the simple closed curve $\partial C$, 
composed of a subcurve $P_1$ of~$P$ and a part $P_2$ of edge $ab$ between the crossings, encloses a simply connected region $C$, see Figure~\ref{fig:Caratheodory_illustration_page1}.

\begin{figure}[tb]
  \centering
  
  \hbox{}\hfill
    \begin{subfigure}[t]{.35\textwidth}
    \centering
    \includegraphics[page=1]{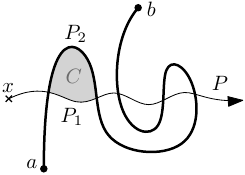}
    \caption{}
    \label{fig:Caratheodory_illustration_page1}  
  \end{subfigure}
  \hfill
  \begin{subfigure}[t]{.35\textwidth}
    \centering
    \includegraphics[page=2]{Caratheodory_illustration_new}
    \caption{}
    \label{fig:Caratheodory_illustration_page2}  
  \end{subfigure}
  \hfill\hbox{}

  \caption{
  \subref{fig:Caratheodory_illustration_page1} and \subref{fig:Caratheodory_illustration_page2} give
  an illustration of the proof of Theorem~\ref{theorem:Caratheodory_generalized}. 
  }
  \label{fig:Caratheodory_illustration}
\end{figure}

Now consider the curve $P'$ from $x$ to infinity which is obtained from $P$ by
replacing the subcurve $P_1$ by a curve $P_2'$ which is a close copy of $P_2$
in the sense that it has the same crossing pattern with all edges in $D$ and
the same topological properties, but is disjoint from $ab$.  As $P$ was chosen
minimal with respect to the number of crossings with~$ab$, there has to be an
edge of the drawing $D'$ that intersects $P_2'$ (and by the choice of~$P_2'$
also $P_2$).  
This edge has no crossing with~$P$, by construction, and crosses $ab$
at most once, so it has one of its endpoints inside the lense $C$ and one
outside~$C$.  
Depending on whether $b\in C$ or not, we choose an endpoint $c_1$
of that edge such that the edge $bc_1$ in~$D'$ intersects $\partial C$.  But
since they are adjacent, $bc_1$ cannot intersect $ab$ and by the choice of~$P$
it does not intersect $P$.  The contradiction shows that $P$ crosses $ab$ in a
unique point~$p$.

If $a$ has another neighbor $c_2$ in the drawing $D'$ 
then, since only edges incident to $a$ have been removed
there is an edge connecting $b$ to $c_2$ in $D'$. 
The edges $ac_2$ and $bc_2$ do not cross $P$, 
so $x$ is in the interior of the triangle $abc_2$
and we are done.

If there is no edge $ac_2$ in $D'$, then 
$\deg(a)=1$ in $D'$. 
As~$x$ is not in the outer cell of~$D'$, 
there must be an edge $cd$ in $D'$ 
which intersects the partial segment of the edge~$ab$ 
starting in $a$ and ending in $p$, in its interior. Let $c$ be the point on the same side of $ab$ as~$x$;
see Figure~\ref{fig:Caratheodory_illustration_page2}. 
The edges $bc$ and $bd$ of $D'$ cross neither $P$ nor $ab$. 
Consequently, the triangle $bcd$ (drawn blue) must contain $a$ in its interior.
We claim that
the edge $ac$ in the original drawing $D$ (drawn red dashed) 
lies completely inside the triangle $bcd$:
The bounded region defined by the edges $ab$, $cd$,
and $bd$ of $D'$ contains $a$ and $c$. 
Since $D$ is a topological drawing, and $ac$ has no crossing with $ab$ and
$cd$, $ac$ has no crossing with $bd$. This proves the claim.
Now the curve $P$ does not intersect $ac$,
and the only edge of the triangle $abc$ intersected by $P$ is~$ab$.
Therefore, $x$ lies in the interior of the triangle $abc$.
This contradicts the assumption that $(D,x)$ is a counterexample. 
\end{proof}

\subsection{Colorful Carath\'eodory Theorem}
\label{sec:CCT}

B{\'a}r{\'a}ny \cite{Barany1982} generalized Carath\'eodory's theorem as follows:
Given finite point sets $P_0,\ldots,P_d$  from $\RR^d$ 
such that there is a point $x \in \conv(P_0) \cap \ldots \cap \conv(P_d)$,
then $x$ lies in a simplex 
spanned by $p_0 \in P_0,\ldots,p_d \in P_d$. Such a simplex is called \emph{colorful}.
The theorem is known as the \emph{Colorful Carath\'eodory theorem}.

A strengthening, known as the \emph{Strong Colorful Carath\'eodory theorem},
was shown in
\cite{HolmsenPachTverberg2008} and independently in \cite{Arocha2009} (cf.\ \cite{KalaiBlog2009}):
It is sufficient if there is a point~$x$ with $x \in \conv(P_i \cup P_j)$ for all $i \neq j$,
to find a colorful simplex.
The Strong Colorful Carath\'eodory theorem was further generalized to oriented matroids by
Holmsen~\cite{Holmsen2016}.
In particular, the theorem applies to pseudolinear drawings 
(which are in correspondence with oriented matroids of rank~3).

Holmsen's proof~\cite{Holmsen2016} uses
sophisticated methods from topology.
We have convinced ourselves that B{\'a}r{\'a}ny's proof~\cite{Barany1982} 
can be adapted to pseudoconfigurations of points in the plane.
Instead of the Euclidean distance, one can
use a discrete metric that counts the minimum number of cells to traverse.
However, B{\'a}r{\'a}ny's proof idea does not directly generalize to higher dimensions because
oriented matroids of higher ranks do not necessarily 
have a representation in terms of pseudoconfigurations of points in $d$-space 
(cf.\ \cite[Chapter~1.4]{BjoenerLVWSZ1993}). 

\medskip

The following result shows that in the convexity hierarchy of topological drawings of $K_n$
the Colorful Carath\'eodory theorem is not valid beyond the class of pseudolinear drawings.

\begin{figure}[tb]
  \centering

    \includegraphics[page=1]{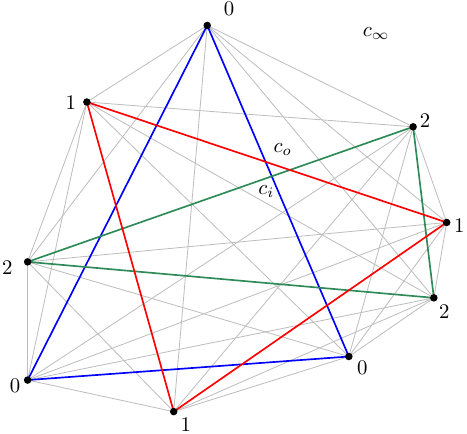}
    
  \caption{
  A cell-convex drawings of $K_9$.
  If the cell $c_o$ is chosen as the outer cell, 
  then Colorful Carath\'eodory theorem does not hold for the colored triangles and every point $x$ from the cell $c_i$.
  The special cell of the pseudolinear drawing is marked $c_\infty$. 
  }
    \label{fig:n9_CCT_counterexample}  
\end{figure}
\begin{proposition}
	The Colorful Carath\'eodory theorem does not hold for
	the cell-convex drawing of Figure~\ref{fig:n9_CCT_counterexample}.
\end{proposition}

\begin{proof}
The drawing depicted in Figure~\ref{fig:n9_CCT_counterexample} 
is cell-convex because it is obtained from a straight-line drawing 
by choosing $c_o$ as outer cell. Let $x$ be an arbitrary point from $c_i$.
The point $x$ is contained in the three colored triangles and is separated from the outer cell
only by three colored edges. 
Therefore, there is no
triangle in the drawing spanned by a red, a green, and a blue point such that $x$ is contained in the triangle formed by these three points.
\end{proof}

\section{Helly's Theorem}
\label{sec:Helly}

The \emph{Helly number} of a family of sets $\FF$ with  empty intersection
is the size of the smallest subfamily of $\FF$ with  empty intersection.
\emph{Helly's theorem} asserts that 
the Helly number of a family of $n$ convex sets 
$S_1,\ldots,S_n$ from $\RR^d$ is at most~$d+1$,
i.e.,
the intersection of $S_1,\ldots,S_n$ is non-empty
if the intersection of every subfamily of size $d+1$ is non-empty.

In the following we discuss the Helly number
in the context of topological drawings,
where the sets $S_i$ are triangles of the drawing.

From the results of Goodman and Pollack \cite{GoodmanPollack1982} it follows
that Helly's theorem generalizes to pseudoconfigurations of points in two
dimensions, and thus for pseudolinear drawings.  A more general version of
Helly's theorem was shown by Bachem and Wanka \cite{BachemWanka1988}.  They
prove Helly's and Radon's theorem for oriented matroids with the
``intersection property''.  Since all oriented matroids of rank~3 have the
intersection property (cf.\ \cite{BachemWanka1988} and \cite{BachemWanka1989})
and oriented matroids of rank~3 correspond to pseudo\-configurations of
points, which in turn yield pseudolinear drawings, the two theorems are valid for
pseudolinear drawings.

\medskip

We show that Helly's theorem does not hold for cell-convex drawings,
moreover, the Helly number can be arbitrarily large in cell-convex drawings.
Note that the following proposition 
does not contradict the topological Helly theorem \cite{Helly1930} (cf.\ \cite{GoaocPPTW2017}) 
because 
in our construction the number of connected components of the intersection can grow arbitrarily large. 
More precisely, if the number of triangles $n$ is even,
the intersection of the $n/2$ triangles with even index 
has $n/2$ connected components
 (see Figure~\ref{fig:helly_counterexample} for an illustration).

\begin{proposition}
	\label{theorem:helly_fconvex}
	Helly's theorem does not generalize to cell-convex drawings. 
	Moreover, for every integer $n \ge 3$,
	there exists a cell-convex drawing of $K_{3n}$
	with Helly number at least~$n$, i.e.,
	there are $n$ triangles such that the bounded sides of any $n-1$ triangles have a common interior point, 
	but the intersection of the bounded sides of all~$n$ triangles is empty.
\end{proposition}

\begin{proof}
	Consider a straight-line drawing $D$ of $K_{3n}$ with $n$ triangles $T_i$ as
	shown for the case $n=7$ in Figure~\ref{fig:helly_counterexample}.
	With $D'$ we denote the drawing obtained from~$D$ by making the gray
	cell $c_o$ the outer cell. 
	Let $O_i$ be the side of $\partial T_i$ that is bounded in~$D'$.
	For $1\leq i < n$ the set $O_i$
	corresponds to the outside of $\partial T_i$ in $D$  while $O_n$ corresponds
	to the inside of~$\partial T_n$.
	
	In $D'$ we have  $\bigcap_{i=1}^{n-1} O_i \neq\emptyset$,  indeed
	any point $p_n$ which belongs to the outer cell of $D$ is in this intersection.
	Since $T_n \subset \bigcup_{i=1}^{n-1} T_i$, we have $T_n \cap \bigcap_{i=1}^{n-1} O_i = \emptyset$,
	i.e., $\bigcap_{i=1}^{n} O_i = \emptyset$.
	For each
	$i \in \{1,\ldots,n-1\}$ 
	there is a point $p_i \in T_i \cap T_n$
	which is not contained in any other $T_j$. Therefore,
	$p_i \in \bigcap_{j=1; j\neq i}^n O_i$.
	
	In summary, the intersection of any $n-1$ of the $n$ sets
	$O_1,\ldots,O_n$ is non-empty but the intersection of all of them
	is empty.
\end{proof}

\begin{figure}[tb]
	\centering
        \includegraphics{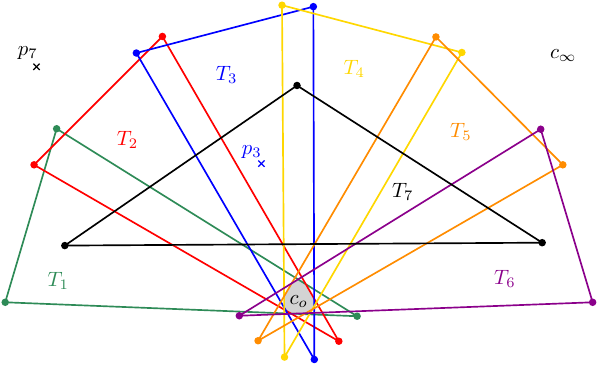}	
	\caption{
		A drawing $D$ of $K_{21}$ is obtained by adding the 
		remaining edges as straight-line segments. Making the gray
                cell $c_o$ the outer cell, we obtain a cell-convex drawing
                with Helly number~$7$.
	}
	\label{fig:helly_counterexample}  
\end{figure}

\section{Generalized Signotopes: Structure and Enumeration }
\label{sec:GS_counting}

In this section we 
discuss the connection between generalized signotopes and topological drawings.
We show that the number of generalized signotopes on $n$ elements is of order $2^{\Theta(n^3)}$.
By introducing a notion of flips
for generalized signotopes,  
we show that generalized signotopes indeed are a proper
generalization of topological drawings 
and estimate how many generalized signotopes 
can be represented by a topological drawing.  
From the known estimates for the asymptotic
number of topological drawings, it then follows that most generalized
signotopes do not come from topological drawings.

\subsection{Flip-equivalent Generalized Signotopes}

Let $\chi$ be a generalized signotope on $[n]$.
A pair $(i,j)$ of distinct elements of $[n]$ is said to be
\emph{flippable} in $\chi$ if inverting the signs of the triples containing
$i$ and $j$ yields a generalized signotope.
If $\chi$ comes from a topological drawing and $(i,j)$ is an edge incident to
the outer cell, then $(i,j)$ is flippable in $\chi$. Moreover, the generalized
signotope $\chi'$ obtained by inverting all triples containing $i$ and $j$
again comes from a drawing.  Indeed, if $D$ is a drawing corresponding to
$\chi$ and the edge $e=(i,j)$ is incident to the outer cell $c_1$, then there is a
second cell $c_2$ which is separated from $c_1$ only by $e$.  Using
stereographic projections, one can wrap the edge $e$ around the drawing to
make $c_2$ the outer cell. The drawing $D'$ obtained this way corresponds to
$\chi'$. \TypeIa and \typeIb of the topological drawings of $K_4$ (see Figure~\ref{fig:k4_three_types}) differ by such a flip operation applied to the edge
$(3,4)$.

Two generalized signotopes $\chi,\chi'$ are \emph{flip-equivalent} 
if there is a sequence $(i_1,j_1)$, \ldots, $(i_k,j_k)$ of pairs and a sequence
$\chi_0,\ldots,\chi_k$ of generalized signotopes with 
$\chi=\chi_0$, $\chi' = \chi_k$, and $\chi_\ell$ is obtained from  $\chi_{\ell-1}$
by flipping the pair $(i_\ell,j_\ell)$.
This flip-equivalence relation partitions 
the set of all generalized signotopes into \emph{flip classes}, 
which we further consider to be closed under relabeling of the elements.

In the following, we show that two weakly isomorphic drawings yield flip-equivalent generalized signotopes. In fact, the following lemma will be the key to show that most generalized signotopes do not come from topological drawings.

\begin{lemma}
	\label{lemma:weak_iso_flip_equiv}
	Two weakly isomorphic drawings 
	$D$ and $D'$ of $K_n$ 
	yield flip-equivalent generalized signotopes.
\end{lemma}
\begin{proof}
	According to \cite{Gioan05_full}, we can transform $D$ into
	$D'$ using triangle-flips (Figure~\ref{fig:weak_strong_isomorphism}) only. Suppose we have $D=D_0,D_1,\ldots,D_m=D'$,
	where $D_i$ is transformed into $D_{i+1}$ by a triangle-flip.  We have
	to show that $\chi(D_i)$ and $\chi(D_{i+1})$ are flip-equivalent generalized
	signotopes. A crucial point is that generalized signotopes come from
	drawings in the plane while weak isomorphism is a property of
	spherical drawings. Hence, we have to allow triangle-flips with the
	triangle being the outer cell. 
	
	Let $\chi(D_i)$ be the generalized signotope of the drawing $D_i$ and let
	$\triangle_i$ be the triangular cell in $D_i$ which is flipped to obtain
	$D_{i+1}$.  If $\triangle_i$ is a bounded cell, we are done because of
	$\chi(D_i) = \chi(D_{i+1})$. Otherwise, if $\triangle_i$ is the outer cell, 
	apply an flip of an edge bounding the outer cell to obtain an isomorphic
	drawing $D_i'$ in which another cell is the outer cell.  Because of the
	edge-flip, $\chi(D_i')$ is flip-equivalent to $\chi(D_i)$.
\end{proof}

\subsection{Small Configurations}
\label{ssec:small_configurations}


\begin{table}[tb]\centering
\advance\tabcolsep5pt
\def\arraystretch{1.2}

\begin{tabular}{ r||rrr|r }
	& Gen.Sig. 
	& Relabeling Cl. 
	& Flip Cl.
	& Weak Isom. Cl.\\
	\hline
	\hline
	3        &2                   &1       &1    &1 \\
	4        &14                  &2       &2    &2 \\
	5        &544                 &6       &3    &5 \\
	6        &173 128             &167     &16   &102 \\
	7        &630 988 832         &63 451  &442  &11 556 \\
	8        &?                   &?       &?    &5 370 725 \\    
	9        &                    &        &     &7 198 391 729 \\
	\vdots   &                    &        &     & \\
	$n$      &$2^{\Theta(n^3)}$   &$2^{\Theta(n^3)}$ &$2^{\Theta(n^3)}$ & $2^{\Theta^*(n^2)}$\\
\end{tabular}

\caption{
The first three columns show 
the number of generalized signotopes on $n$ elements,
equivalence classes up to relabeling, and flip classes, respectively.
The last column shows 
the number of weak isomorphism classes of topological drawings of $K_n$ from \cite{aafhpprsv-agdsc-15} (cf.\ \cite{Pammer2014} and \href{https://oeis.org/A276110}{OEIS/A276110}).
The asymptotic bounds 
are provided in Theorem~\ref{thm:generalized_signotopes_bound} and \cite{Kyncl2013,PachToth2006}, respectively.
}

\label{table:GS_numbers}
\end{table}


To get a better understanding of 
which generalized signotopes come from topological drawings,
we have 
enumerated all generalized signotopes 
and flip classes up to $n=7$ elements 
using a simple computer program;
see Table~\ref{table:GS_numbers}.
Moreover, since drawings from the same weak isomorphism class 
induce flip-equivalent generalized signotopes
(Lemma~\ref{lemma:weak_iso_flip_equiv}),
Table~\ref{table:GS_numbers} also restates the number of weak isomorphism classes from \cite{aafhpprsv-agdsc-15}.

In Section~\ref{ssec:prelim_GS},
we have seen that there are precisely two weak isomorphism classes of topological drawings of~$K_4$.
Via relabeling and mirroring,
all 14 generalized signotopes on $n=4$ elements are realized by \typeI and \typeII;
cf.\ Figure~\ref{fig:k4_three_types}.
These 14 generalized signotopes 
partition into two relabeling classes
and two flip classes. One of the classes corresponds to drawings of \typeI (1 crossing) and
the other one to drawings of \typeII (0 crossings).
In particular, the flip operation for generalized signotopes preserves the number of crossings in a 4-tuple.
Therefore, we can define the \emph{crossing number} of a generalized signotope $\chi$ on $n$ elements 
as the number of induced 4-tuples which belong to the flip class of the \typeI drawing.

\begin{figure}[tb]
	\centering
	\includegraphics[]{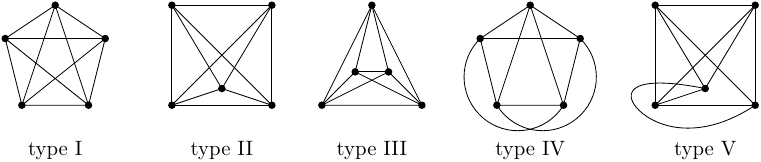}
	\caption{The five types of topological drawings of $K_5$.}
	\label{fig:n5_fivetypes}  
\end{figure}

For $n=5$,
there are 544 generalized signotopes,
which belong to 6 relabeling classes and 3 flip classes, respectively.
There are five weak isomorphism classes of topological drawings of $K_5$, 
see Figure~\ref{fig:n5_fivetypes}.
We have verified by computer that each of the 544 generalized signotope on $n=5$ elements 
is realized by a topological drawing of~$K_5$. 
Since we can read
whether a 4-tuple of vertices induces a crossing from the generalized signotope,
it is clear that
drawings with different number of crossings do not correspond to a common class
of generalized signotopes.
Indeed, the class with 24 generalized signotopes
corresponds to type~I  and type~V  (both 5 crossings), 
the class with 280 generalized signotopes
corresponds to type~II and type~IV (both 3 crossings), 
and the class with 240 generalized signotopes
corresponds to type~III  (1 crossing).
We conclude that generalized signotopes are not able to encode
the weak isomorphism class. 
Also convexity is not encoded:
type~I and type~V 
induce the same generalized signotope but type~I is cell-convex
while type~V is non-convex. 

For $n=6$,
there are 173\,128 generalized signotopes,
167 relabeling classes, and 
16 flip classes.
We have verified by computer  
that 151 of the 167 relabeling classes 
are realized by a topological drawing of~$K_6$.
However, from each of the 16 flip classes 
there is a representative which can be realized by a topological drawing of $K_6$.

The non-realizable generalized signotopes on $n=6$
belong to three flip classes, which have 3, 4, and 5 crossings, respectively.
Note that there is a unique flip class with 3 crossings, 
a unique flip class with 4 crossings,
and two flip classes with 5 crossings.

We now consider the flip class $F$ 
of generalized signotopes on $n=6$ elements with 3 crossings.
There is, up to strong isomorphism, a unique
topological drawing $D$ of $K_6$ 
which has the minimum of 3 crossings;
see Figure~\ref{fig:n6_crossingmin}. 
Therefore, every drawing realizing a generalized signotope from $F$ is isomorphic to~$D$.
Since the drawing $D$ is highly symmetric,
there are up to isomorphism only 3 choices for the outer cell,
and hence only 3 of the 10 generalized signotopes from the flip class $F$ are realized; cf.\ Listing~\ref{lst:gensig6real}.
The remaining 7 generalized signotopes of that flip class are not realizable; cf.\ Listing~\ref{lst:gensig6nonreal}.
Note that in Listings~\ref{lst:gensig6real} and~\ref{lst:gensig6nonreal}
we encode a generalized signotope $\chi$ on the elements $\{1,\ldots,6\}$
only by its $+$-triples, that is, the pre-image $\chi^{-1}(+)$.

\begin{lstlisting}[caption={Three realizable generalized signotopes 
on the elements $\{1,2,3,4,5,6\}$ from the flip class~$F$
encoded by its $+$-triples.},label={lst:gensig6real}]
{235,236,245,246,345,346,356,456}
{235,236,245,246,256,345,346,356}
{234,235,245,246,256,346,356,456}
\end{lstlisting}

\begin{lstlisting}[caption={Seven non-realizable generalized signotopes  
on the elements $\{1,2,3,4,5,6\}$ from the flip class~$F$ 
encoded by its $+$-triples.},label={lst:gensig6nonreal}]
{234,235,236,245,256,346,356,456}
{234,235,236,246,256,345,356,456}
{234,235,246,256,345,346,356,456}
{136,234,245,256,345,456}
{234,236,245,246,256,345,356,456}
{234,236,245,256,345,346,356,456}
{235,236,245,246,256,345,346,456}
\end{lstlisting}

\begin{figure}[htb]
  \centering
  \includegraphics[]{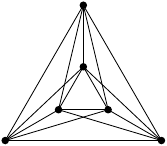}
  \caption{The unique topological drawing of $K_6$ 
  which has the minimum of 3~crossings.}
  \label{fig:n6_crossingmin}  
\end{figure}

To lift the non-representable examples to higher number of elements we
use the all-plus-extension of a generalized signotope.

\begin{lemma}[All-plus-extension]\label{lemma:all-plus-extension}
Let $\chi$ be a generalized signotope on $n$ elements and let $n' \ge n$ be an integer.
Then the mapping $\chi': [n']_3 \to \{+,-\}$ with
	\[
	\chi'(x,y,z) = 
	\begin{cases}
	\chi(x,y,z) &\text{ if $x,y,z \in [n]$}\\
	+ &\text{ otherwise.}\\
	\end{cases}
	\]  
is a generalized signotope on $n'$ elements.
\end{lemma}

\begin{proof}
Consider four elements $x,y,z,w \in [n']$.
If $x,y,z,w \in [n]$,
then the sequence $\chi'(xyz),\allowbreak{}\chi'(xyw),\allowbreak{} \chi'(xzw),\allowbreak{} \chi'(yzw)$ avoids 
the forbidden patterns $+-+-$ and $-+-+$ 
because $\chi$ is a generalized signotope.
Otherwise,
the sequence $\chi'(xyz),\allowbreak{}\chi'(xyw),\allowbreak{} \chi'(xzw), \allowbreak{}\chi'(yzw)$
contains at least three $+$-entries.
\end{proof}

\begin{corollary}
For $n \le 5$ all generalized signotopes on $n$ elements 
are realizable as topological drawing of $K_n$.
For $n \ge 6$ there exist non-realizable generalized signotopes on $n$ elements.
\end{corollary}

\begin{proof}
The first part follows from earlier discussions in this subsection.

For the second part,
consider the non-realizable generalized signotope $\chi$ on $6$ elements from above.
Now, for every integer $n'$ with $n' \ge 6$,
the all-plus-extension of $\chi$ (Lemma~\ref{lemma:all-plus-extension})
is also non-realizable since it contains $\chi$ as an induced subconfiguration. 
\end{proof}

Another interesting example is the generalized signotope on $n=7$ elements
shown in Listing~\ref{lst:gensig7cr7}.  This configuration is not
representable by any topological drawing of~$K_7$ because it has crossing
number~7 while every drawing of $K_7$ has at least~9 crossings \cite{Guy1972}.

\begin{lstlisting}[caption={A generalized signotope
on the elements $\{1,2,3,4,5,6,7\}$ with only 7~crossings
encoded by its $+$-triples.},label={lst:gensig7cr7}]
{235,236,237,245,246,247,257,267,345,346,347,356,367,456,457,567}
\end{lstlisting}

It would be interesting to have non-trivial bounds for the minimum number of crossings
of generalized signotopes on $n$ elements.

\subsection{The Asymptotic Number}
\label{ssec:num-gs}

In this subsection we show that 
the number $g(n)$ of generalized signotopes on $n$ elements
is of order $2^{\Theta(n^3)}$.
This bound also applies to the numbers of 
relabeling classes and flip classes, respectively,
because reflections and relabelings only give a factor of at most 
$2\cdot n!$
and the number of elements in a flip class is at most $2^{\binom{n}{2}}$.
Last but not least, we show that most generalized signotopes are not induced by a topological drawing.

\paragraph{Upper Bound for $g(n)$:}
    To eventually show that 
    $g(n) \le g(t)^{\binom{n}{t}/\binom{n-3}{t-3}}$ is an upper bound on the number of generalized signotopes on $n$ elements,
    we make use of Shearer's Entropy Lemma \cite{ChungGFS1986}. 
	
	\begin{lemma}[Shearer's Entropy Lemma, \cite{ChungGFS1986}]
	\label{lemma:shearer_entropy_lemma}
    Let $S$ be a finite set 
    and let $A_1,\ldots,A_m$ be subsets of $S$ such that
    every element of $S$ is contained in at least $k$ of the sets $A_1,\ldots,A_m$.
    If~$\FF$ is a collection of subsets of $S$ and 
    $\FF_i = \{F \cap A_i \colon F \in \FF \}$ for $1 \le i \le m$.
    Then
    \[
     |\FF|^k \le \prod_{i=1}^m |\FF_i|.
    \]
	\end{lemma}
    
        Let $t \leq n$.  We consider the set $S = \binom{[n]}{3}$ of all triples
        from $[n]$ and, for each $t$-subset~$I$ of~$[n]$,
        let $A_I = \binom{I}{3}$ be the set of triples of~$I$.  There are
        $m=\binom{n}{t}$ choices for $I$ and as many sets $A_I$.  Each triple
        in $S$ belongs to $k=\binom{n-3}{t-3}$ sets~$A_I$.
    
        A generalized signotope on $n$ elements is uniquely encoded by its
        $+$-triples, which form a subset of~$S$.  Let $\FF$ be the family of
        all generalized signotopes on $n$ elements given by their $+$-triples.  For every
        $I$, let $\FF_I = \{F \cap A_I \colon F \in \FF \}$. Note that $\FF_I$
        is a family of generalized signotopes on $I$, whence
        $|\FF_I| \leq g(t)$.

    
    Lemma~\ref{lemma:shearer_entropy_lemma} 
    implies 
    \[
        g(n)^k = |\FF|^k \le \prod_{I \in \binom{[n]}{t}} |\FF_I| \le g(t)^m,
      \]
      
    with $m=\binom{n}{t}$ and $k=\binom{n-3}{t-3}$. Therefore,
    \[
        g(n) \le g(t)^{m/k} = 2^{c(t)\binom{n}{3}}
    \quad
     \text{with}
    \quad   
        c(t) = {\log_2(g(t))}/{\binom{t}{3}}.
    \]
    
    Using $g(7)=630\ 988\ 832$ (cf.\ Table~\ref{table:GS_numbers}),
	we obtain that
	the number $g(n)$ of generalized signotopes on $n$ elements 
	is at most $2^{c_2 \cdot \binom{n}{3}}+o(1)$ 
	where $c_2=c(7) \approx 0.8352$. 

Note that the above shows that $c(n)\leq c(t)$, that is, $c$ is non-increasing. Thus, the factor $c_2$ can be expected to decrease if a value of $c(t')$ with $t' > 7$ becomes available.

	\paragraph{Lower Bound for $g(n)$:}
	First, we give a recursive construction of a set $\XX_{3n}$ 
	of generalized signotopes on $3n$ elements.
	The set $\XX_3$ consists of the two generalized signotopes on $\{1,2,3\}$.
	
	For the step, we construct $\XX_{3n}$ based on $\XX_{n}$:
	Let $A=\{1,\ldots,n\}$, $B=\{n+1,\ldots,2n\}$, and $C=\{2n+1,\ldots,3n\}$.
	Pick three generalized signotopes $\chi_A$, $\chi_B$, $\chi_C$ from $\XX_{n}$ 
	and an arbitrary mapping $M:A \times B \times C \to \{+,-\}$. 
	We define $\chi$ by the following rule: 
	for $x<y<z$ we set
	\[
	\chi(x,y,z) = 
	\begin{cases}
	\chi_A(x,y,z) &\text{ if $x,y,z \in A$}\\
	\chi_B(x,y,z) &\text{ if $x,y,z \in B$}\\
	\chi_C(x,y,z) &\text{ if $x,y,z \in C$}\\
	M(x,y,z) &\text{ if $x \in A, y \in B,z \in C$}\\
	+ &\text{ otherwise.}\\
	\end{cases}
	\]  
	An easy case distinction shows that $\chi$ is a generalized signotope on $n$ elements: 
	For any four elements $x<y<z<w$, at least two are from the same class $S\in\{A,B,C\}$.
	We look at the signs of the sequence $xyz,xyw,xzw,yzw$.

        If all four elements are from $S$, then we use that $\chi_S$ is a generalized signotope.
        If exactly three of the elements are from $S$, then there are at least three $+$ signs
        in the sequence, whence, the forbidden patterns $+-+-$ and $-+-+$ do not occur.
	Now if exactly two of the elements are from $S$, then
        if the two elements are $x,y$ the triples $xyz$ and $xyw$ map to plus and we have $++**$, 
	where $*\in \{+,-\}$ is arbitrary,
	if $y,z$ are from $S$, we have $+**+$, and if $z,w$ are from $S$, we have $**++$.
	In any case, the forbidden patterns $+-+-$ and $-+-+$ cannot occur, 
	and hence $\chi$ is a generalized signotope.
	
	Since there are $|\XX_{n}|^3 \cdot 2^{n^3}$ possibilities to choose $\chi_A,\chi_B,\chi_C,M$, 
	and no two such selections yield the same $\chi$,
	we have 
	\[
	|\XX_{3n}| = |\XX_{n}|^3 \cdot 2^{n^3}.
	\]

Now, using all-plus-extensions (cf.\ Lemma~\ref{lemma:all-plus-extension}), 
we obtain sets $\XX_{3n+1}$ and $\XX_{3n+2}$
of generalized signotopes on $3n+1$ and $3n+2$ elements, respectively,
with $|\XX_{3n}|=|\XX_{3n+1}|=|\XX_{3n+2}|$.
Hence,  for $f(n) = \log_2 |\XX_n|$ we have 
\[
f(n) = 3f(\lfloor n/3 \rfloor)+\lfloor n/3 \rfloor^3.
\]
Inductively assuming 
$
f(n) \ge \frac{1}{24}n^3 - \frac{3}{8}n^2
$,
which is easy to check for $n=1$ and $n=2$,
we obtain 
\begin{align*}
f(n) 
& = 3f \left( \left\lfloor \frac{n}{3} \right\rfloor \right)+ \left\lfloor \frac{n}{3} \right\rfloor^3 
\\
& \ge
3\left( \frac{1}{24} \cdot \left\lfloor \frac{n}{3} \right\rfloor^3 - \frac{3}{8} \cdot \left\lfloor \frac{n}{3} \right\rfloor^2 \right) 
+\left\lfloor \frac{n}{3} \right\rfloor^3 
\\
& \ge
3\left( \frac{1}{24} \cdot \left(\frac{n-2}{3}\right)^3 - \frac{3}{8} \cdot \left(\frac{n-2}{3}\right)^2 \right) 
+\left(\frac{n-2}{3}\right)^3 
\\
& =
\frac{1}{24} n^3
- \frac{3}{8}n^2
+ n
- \frac{5}{6}
\ 
\ge
\ 
\frac{1}{24} n^3
- \frac{3}{8}n^2
\end{align*}
for every $n\ge 3$. \goodbreak

We summarize the results in the following theorem.

\begin{theorem}\label{thm:generalized_signotopes_bound}
	The number $g(n)$ of generalized signotopes on $n$ elements is between 
	$2^{c_1 \cdot \binom{n}{3}+o(n^3)}$ 
	and 
	$2^{c_2 \cdot \binom{n}{3}}+o(1)$ 
	for constants $c_1 = 0.25$ and $c_2 \approx 0.8352$.
\end{theorem}

Last but not least,
we investigate how many generalized signotopes come from topological drawings.
There are at most $2^{\tilde{O}(n^2)}$ weak isomorphism classes of drawings of the complete graph $K_n$
\cite{Kyncl2013} (cf.\ \cite{PachToth2006}) and, 
by Lemma~\ref{lemma:weak_iso_flip_equiv},
each weak isomorphism classes is contained in a flip-equivalence class of generalized signotopes.
Since the number of generalized signotopes in a flip-equivalence class
is at most $ 2^{\binom{n}{2}}$,
we conclude that
at most $2^{\tilde{O}(n^2)} \cdot 2^{n \choose 2}=2^{\tilde{O}(n^2)}$ 
generalized signotopes come from topological drawings of~$K_n$.

\section{Discussion} 
\label{sec:Discussion}

We conclude this article 
with remarks on three additional classic theorems from Convex Geometry. 

Lov\'asz (cf.\ B{\'a}r{\'a}ny \cite{Barany1982}) 
generalized Helly's theorem as follows:
Let $\mathcal{C}_0,\ldots,\mathcal{C}_d$ be families of compact convex sets from $\RR^d$ 
such that for every ``colorful'' choice of sets $C_0 \in \mathcal{C}_0,\ldots,C_d \in \mathcal{C}_d$
the intersection $C_0\cap \ldots \cap C_d$ is non-empty.
Then, for some $k$, the intersection $\bigcap \mathcal{C}_k$ is non-empty.
This result is known as the \emph{Colorful Helly theorem}.
Kalai and Meshulam \cite{KalaiMeshulam2005} presented a topological version of
the Colorful Helly theorem, which, in particular, carries over to pseudolinear
drawings.  Since Helly's theorem does not generalize to cell-convex drawings
(cf.\ Proposition~\ref{theorem:helly_fconvex}), 
neither does the Colorful Helly theorem.

The \emph{$(p,q)$-Theorem} 
(conjectured by Hadwiger and Debrunner, 
proved by Alon and Kleitman \cite{AlonKleitman1992}, 
cf.\ \cite{KellerSmorodinskyTardos2018})
says that for any $p \ge q \ge d+1$ 
there is a finite number $c(p,q,d)$ with the
following property:
If $\mathcal{C}$ is a family of convex sets in $\RR^d$, 
with the property that among any $p$ of them,
there are $q$ that have a common point, 
then there are $c(p,q,d)$ points that cover all the sets in~$\mathcal{C}$.
Helly's theorem is the case with $p=q=d+1$, i.e., $c(d+1,d+1,d)=1$.
We are not aware whether a $(p,q)$-Theorem for triangles in general topological drawings exists, however,
an anonymous reviewer pointed us to a proof for pseudolinear drawings.
Here is an outline:
A triangle in a pseudolinear drawing is the intersection of three pseudo-halfplanes. Hence, the intersection of multiple triangles is the intersection of pseudo-halfplanes, and is therefore either empty or path-connected.
A $(p,q)$-Theorem for triangles in pseudolinear drawings now follows directly from Pat{\'a}kov{\'a}'s $(p,q)$-Theorem \cite[Theorem~6]{Patakova2020}. 

Last but not least, we would like to mention 
\emph{Tverberg's theorem}, which asserts that 
every set $V$ of at least $(d+1)(r-1)+1$ points in $\RR^d$
can be partitioned into $V=V_1 \mathbin{\dot{\cup}} \ldots \mathbin{\dot{\cup}} V_r$
such that $\conv(V_1) \cap \ldots \cap \conv(V_r)$ is non-empty.
A generalization of Tverberg's theorem applies to pseudolinear drawings~\cite{Roudneff1988},
to drawings of $K_{3r-2}$ if $r$ is prime \cite{BaranyiShlosmanSzucs1981} and if $r$ is a prime-power \cite{Ozaydin1987}. 
In particular, every drawing of the $K_4$ provides a partition into two sets with a common intersection. This shows that Radon's theorem holds for topological drawings. 
Also a generalization of Birch's theorem, a weaker version of Tverberg's theorem,
was recently proven for topological drawings of complete graphs \cite{FrickSoberon2020}.
The general case, however, remains unknown.
For a recent survey on generalizations of Tverberg's theorem,
we refer to \cite{BaranySoberon2017}.

Besides the mentioned characterization of pseudolinear drawings \cite[Theorem~3.2]{BalkoFulekKyncl2015},
Balko, Fulek, and Kyn\v{c}l also
provide a characterization of
which generalized signotopes
can be drawn as $x$-monotone topological drawings
and which can be drawn as $x$-monotone semisimple drawings
by forbidding finitely many subconfigurations
\cite[Theorem~3.1]{BalkoFulekKyncl2015}.
In the spirit of their results,
Kyn\v{c}l's theorem \cite{Kyncl2020},
and the hereditary-convex and convex classification by Arroyo et al.\ \cite{ArroyoMRS2017_convex},
we pose the following question on characterizing drawable generalized signotopes: 
\begin{question}
	Is there a finite number $k$ such that,
	given any generalized signotope, if all $k$-tuples are drawable,
	then the generalized signotope is drawable? 
\end{question}

\small
\bibliographystyle{alphaabbrv-url}
\bibliography{references}

\end{document}